\documentclass[a4paper,12pt,reqno]{amsart}  

\usepackage{amsmath}
\usepackage{cases}
\usepackage[all]{xy}
\usepackage{graphicx}
\usepackage{tikz-cd}
\usepackage{tikz}
\usepackage{pgfplots}
\usepackage{amssymb}
\usepackage{amsthm}
\usepackage{amsfonts}
\usepackage{mathtools}
\usepackage{comment}
\usepackage{mathrsfs}
\usepackage[noadjust]{cite} 
\usepackage{enumerate}
\usepackage{here}
\usepackage{wrapfig}
\usepackage{lscape}
\usepackage{longtable} 
\usepackage{comment} 

\usepackage[colorlinks]{hyperref}  
\usepackage{xcolor}
\hypersetup{
	bookmarksnumbered=true,
    colorlinks=true,
    citecolor=blue,
    linkcolor=purple,
    urlcolor=orange,
}
\usepackage{pgfplots}
\pgfplotsset{compat=1.18}

\usepackage[capitalize, nameinlink]{cleveref}

\everymath{\displaystyle}
\numberwithin{equation}{section}

\theoremstyle{plain}

\newtheorem{thm}{Theorem}[section]

\newtheorem{lem}[thm]{Lemma}
\newtheorem{prop}[thm]{Proposition}

\newtheorem*{claim*}{Claim}

\newtheorem{introthm}{Theorem}[section]

\newtheorem{introquestion}[introthm]{Question}
\newtheorem{introprop}[introthm]{Proposition}

\theoremstyle{definition}

\newtheorem{dfn}[thm]{Definition}

\newtheorem{eg}[thm]{Example}
\newtheorem*{Ack}{Acknowledgement}
\newtheorem*{NoCon}{Notation and Conventions}
\newtheorem*{Out}{Outline of this paper}

\newtheorem*{AI}{AI Usage}

\theoremstyle{remark}
\newtheorem{rem}[thm]{Remark}
\newtheorem*{rem*}{Remark}

\DeclareMathOperator{\Pic}{Pic}
\DeclareMathOperator{\Stab}{Stab}
\DeclareMathOperator{\divstab}{Stab_{\mathrm{div}}}
\DeclareMathOperator{\Amp}{Amp}

\DeclareMathOperator{\ch}{ch}

\DeclareMathOperator{\rk}{rk}

\DeclareMathOperator{\Coh}{Coh}

\newcommand{\PPO}{\mathbb{P}^{1}}
\newcommand{\PPT}{\mathbb{P}^{2}}

\newcommand{\wallequation}[2]{\mathcal{W_{\mathrm{eq}}}(#1, #2)}
\newcommand{\wall}[2]{\mathcal{W}(#1, #2)}

\newcommand{\wallpolalim}[5]{\mathcal{W}^{0}_{#3, #4, #5}(#1, #2)}

\newcommand\dual{\raise0.9ex\hbox{$\scriptscriptstyle\vee$}}

\newcommand{\CB}{\mathbb{C}}

\newcommand{\HB}{\mathbb{H}}

\newcommand{\RB}{\mathbb{R}}
\newcommand{\PP}{\mathbb{P}}
\newcommand{\ZZ}{\mathbb{Z}}

\newcommand{\AC}{\mathcal{A}}

\newcommand{\DC}{\mathcal{D}}

\newcommand{\FC}{\mathcal{F}}

\newcommand{\SC}{\mathcal{S}}
\newcommand{\TC}{\mathcal{T}}
\newcommand{\UC}{\mathcal{U}}
\newcommand{\VC}{\mathcal{V}}

\newcommand{\OO}{\mathscr{O}}

\author{Yuki Mizuno}
\address[Y.Mizuno]{Mathematical Institute, Utrecht University, Budapestlaan~6, 3584~CD Utrecht, The Netherlands}
\email{\href{mailto:y.mizuno@uu.nl}{y.mizuno@uu.nl}}

\author{Tomoki Yoshida}
\address[T.Yoshida]{Department~of~Mathematics, School~of~Science~and~Engineering, Waseda~University, Ohkubo~3-4-1, Shinjuku, Tokyo~169-8555, Japan}
\email{\href{mailto:tomoki_y@asagi.waseda.jp}{tomoki\_y@asagi.waseda.jp}}

\title[Bridgeland Stability of Sheaves on del Pezzo Surface]{Bridgeland Stability of Sheaves on del Pezzo Surface of Picard Rank Three}
\date{Nov 20, 2026}


\begin{document}
\begin{abstract}
We study the stability of line bundles and torsion sheaves under divisorial Bridgeland stability conditions.
On the blow-up of $\mathbb{P}^2$ at two general points, we characterize stability and semistability of a line bundle, or of its shift in the tilted heart, in terms of the three $(-1)$-curves.
We also prove that, on any smooth projective surface, the pushforward of a torsion-free sheaf of rank $1$ on an integral curve of nonpositive self-intersection is stable for every divisorial stability condition.
\end{abstract}
\maketitle
\setcounter{tocdepth}{2} 
\tableofcontents
\setcounter{section}{-1}
\section{Introduction}\label{introduction}
If one wishes to investigate sheaves on an algebraic variety $X$, 
one effective approach is to consider the stability of sheaves.
As the derived category of coherent sheaves can be considered as a generalization of the category of sheaves, it is natural to ask whether there is a corresponding notion.
In \cite{bridgeland_2007_stability_conditions_on_triangulated_categories}, Bridgeland introduced \emph{Bridgeland stability conditions} on a triangulated category as an analogue of stability for sheaves.
A numerical Bridgeland stability condition $\sigma=(Z,\AC)$ consists of a group homomorphism $Z:K_{\mathrm{num}}(X)\to\CB$, called the central charge, and the heart $\AC$ of a bounded $t$-structure, satisfying the conditions recalled in \cref{definition: stability conditions on triangulated categories}. 
We denote the space of numerical Bridgeland stability conditions on $X$ by $\Stab(X)$.

While classical sheaf stability on surfaces guarantees, for example, that every line bundle is stable, the Bridgeland stability condition does not necessarily inherit this property. 
For instance, it does not generally hold that fundamental objects, such as skyscraper sheaves or line bundles, remain (semi)stable under arbitrary Bridgeland stability conditions. 
A Bridgeland stability condition under which skyscraper sheaves are stable  of the same phase is called a \emph{geometric stability condition}.

For a smooth projective surface $X$ over $\CB$, there is an injection
\[
N^1(X)_{\RB}\times\Amp(X)_{\RB} \hookrightarrow \Stab(X)
\]
whose image consists of the \emph{divisorial stability conditions} studied in this paper. 
These stability conditions are known to be geometric, and thus can be viewed as a direct generalization of classical sheaf stability.
In this paper, we restrict our attention to divisorial stability conditions.

Even within this geometric framework, line bundles can still fail to be $\sigma$-stable.  
For fixed $D$, ample $H$ and a line bundle $L$, the object $L$ is $\sigma_{D,tH}$-stable for $t\gg0$ by \cite[Sections 3 and 4.2]{arcara_miles_2016_bridgeland_stability_of_line_bundles_on_surfaces}.
We are interested in the stability of line bundles away from this large-volume limit.
On a surface without curves of negative self-intersection, every line bundle is stable for every divisorial stability condition.
This suggests studying how negative divisors detect instability.

At the time of the first version of this paper, our starting point was the following question, which originated from \cite[Conjecture 1]{arcara_miles_2016_bridgeland_stability_of_line_bundles_on_surfaces}.
\begin{introquestion}[{cf.\ \cite[Conjecture 1]{arcara_miles_2016_bridgeland_stability_of_line_bundles_on_surfaces}}]\label[question]{Question: stability question}
    Let $X$ be a smooth projective surface, let $\sigma=\sigma_{D,H}$ be a divisorial stability condition, and let $L$ be a line bundle on $X$.
    \begin{itemize}
        \item If $L\in\AC_{D,H}$ is not $\sigma$-stable, is there a nonzero effective divisor $C$ with $C^2<0$ such that $L(-C)\to L$ is a subobject in the heart and $\beta_\sigma(L(-C))\geq\beta_\sigma(L)$?
        \item If $L[1]\in\AC_{D,H}$ is not $\sigma$-stable, is there such a divisor $C$ for which $L(C)|_C\to L[1]$ is a subobject in the heart and $\beta_\sigma(L(C)|_C)\geq\beta_\sigma(L[1])$?
    \end{itemize}
    In both cases, can $C$ be chosen to be an integral curve?
    
    If semistability fails, can the corresponding inequality be chosen strict?
    Moreover, can $C$ be chosen so that $\beta_\sigma(L(-C))\geq\beta_\sigma(F)$ for every nonzero proper subobject $F$ of $L$ in the heart, and similarly for $L[1]$?
    Finally, is every destabilizing subobject of $L$ (resp.\ $L[1]$) of the form $L(-C)$ (resp.\ $L(C)|_C$)?
\end{introquestion}

The results of \cite{arcara_miles_2016_bridgeland_stability_of_line_bundles_on_surfaces} establish this detection property for surfaces without negative curves and for surfaces of Picard rank $2$ with exactly one irreducible curve of negative self-intersection.

\subsection{Main results}\label{subsection: introduction main theorems}
Our first result proves the detection property for the del Pezzo surface of Picard rank $3$ and identifies three detecting objects.
\begin{introthm}[See \cref{Theorem: stability of line bundle,proposition: stability of shifted sheaf}]\label{Theorem in the introduction: stability of line bundle}
    Let $X$ be the del Pezzo surface of Picard rank $3$, and let $E_1,E_2,E_3$ be its $(-1)$-curves.
    Let $L$ be a line bundle, and let $\sigma=\sigma_{D,H}$ be a divisorial stability condition.
    \begin{enumerate}
        \item If $L\in\AC_{D,H}$, then $L$ is not $\sigma$-stable if and only if, for some $i$, the natural morphism $L(-E_i)\to L$ is a subobject in the heart and $\beta_\sigma(L(-E_i))\geq\beta_\sigma(L)$.
        \item If $L[1]\in\AC_{D,H}$, then $L[1]$ is not $\sigma$-stable if and only if, for some $i$, the natural morphism $L(E_i)|_{E_i}\to L[1]$ is a subobject in the heart and $\beta_\sigma(L(E_i)|_{E_i})\geq\beta_\sigma(L[1])$.
    \end{enumerate}
    In each case, failure of semistability is equivalent to the corresponding condition with a strict inequality.
\end{introthm}

These statements do not assert that one of the three detecting objects has the largest slope among all nonzero proper subobjects.
Indeed, \cref{example: maximality fails} shows that this stronger statement, which was asserted in the first version of this paper, does not hold.
That example does not settle the maximality question when arbitrary negative effective divisors are allowed.

In general, the divisor $C$ in \cref{Question: stability question} cannot be chosen to be an integral curve (\cref{example: irreducible detection fails}).
On the del Pezzo surface of Picard rank $3$, however, \cref{Theorem in the introduction: stability of line bundle} shows that $C$ can be chosen among the three $(-1)$-curves.
If $C$ may be any nonzero effective divisor, possibly reducible or nonreduced, detection holds for arbitrary smooth projective surfaces.
After the first version of this paper, M\"akel\"a established this, including the strict version, in \cite[Theorem 1.1]{preprint_anthony_2026_negative_effective_divisors_and_bridgeland_stability_of_line_bundles_on_surfaces}.
A destabilizing subobject need not be a line bundle, as the last part of \cref{example: maximality fails} also shows.
Our proof of \cref{Theorem in the introduction: stability of line bundle} is independent of \cite{preprint_anthony_2026_negative_effective_divisors_and_bridgeland_stability_of_line_bundles_on_surfaces} and proceeds differently, using the rank $1$ comparison and the wall and rank-reduction arguments in \cref{section: stability}.

Our second result generalizes the stability of torsion sheaves on $(-1)$-curves.
\begin{introprop}[See \cref{proposition: stability of torsion sheaf}]\label[prop]{Proposition in the introduction: stability of torsions}
    Let $Y$ be a smooth projective surface over $\CB$, let $C\subset Y$ be an integral curve with $C^2\leq0$, and let $\iota:C\hookrightarrow Y$ be the inclusion.
    If $L$ is a torsion-free sheaf of rank $1$ on $C$, then $\iota_*L$ is stable for every divisorial stability condition on $Y$.
\end{introprop}
In particular, this includes $\OO_E(k)$ for every $(-1)$-curve $E$ and every $k\in\ZZ$, without any restriction on the Picard rank of the surface.

\subsection{Related work}\label{subsection: introduction future and related works}
Arcara and Bertram established stability of line bundles for divisorial stability conditions $\sigma_{D,H}$ with $D$ a multiple of $H$ \cite{arcara_bertram_2013_bridgelandstable_moduli_spaces_for_ktrivial_surfaces}.
In \cite{arcara_miles_2016_bridgeland_stability_of_line_bundles_on_surfaces}, Arcara and Miles studied divisorial stability conditions $\sigma_{D,H}$ where $D$ is not necessarily a multiple of $H$, by analyzing slices of the space of stability conditions.
As the Picard number increases, the parameter space with fixed ample ray has more directions, although each slice defined by two divisor classes remains three-dimensional.
The results of \cite{arcara_miles_2016_bridgeland_stability_of_line_bundles_on_surfaces} include surfaces without negative curves in arbitrary Picard rank, and the additional Picard rank $2$ cases described above.

As an analogue of the Kobayashi–Hitchin problem (inspired by 
\cite{thomas_2001_moment_maps_monodromy_and_mirror_manifolds} and 
\cite{thomas_yau_2002_special_lagrangians_stable_bundles_and_mean_curvature_flow}), 
\cite{collins_shi_2022_stability_and_the_deformed_hermitianyangmills_equation} posed the following question: for a holomorphic line bundle $L$, is it true that $L$ admits a metric solving the deformed Hermitian–Yang–Mills equation if and only if $L$ is stable for the associated divisorial stability condition?
They proved that any holomorphic line bundle $L$ on the blow-up of $\PP^{2}$ at a point that admits a solution to the deformed Hermitian–Yang–Mills equation is stable for the associated divisorial stability condition, using results of \cite{arcara_miles_2016_bridgeland_stability_of_line_bundles_on_surfaces}. 
However, the converse does not hold in general: 
\cite{collins_shi_2022_stability_and_the_deformed_hermitianyangmills_equation} also constructed a counterexample.
In \cite[Section 6]{preprint_jacopo_2025_nakaimoishezon_criteria_and_the_toric_thomasyau_conjecture}, Stoppa uses the detection statement for the blow-up of $\PP^2$ at two points to relate divisorial Bridgeland stability to deformed Hermitian--Yang--Mills connections and special Lagrangian sections.
That argument uses only the detection statement in \cref{Theorem in the introduction: stability of line bundle}, and not the maximality of a detecting subobject.

\begin{Out}
    In \cref{section: preliminaries}, we recall the numerical stability conditions and distinguish numerical equal-slope loci from actual destabilizing walls.
    In \cref{section: Quiver region and strategy of the proof}, we introduce the slices and coordinates used in the numerical wall calculations.
    \par
    In \cref{section: stability}, we prove \cref{Theorem in the introduction: stability of line bundle}.
    Tensoring reduces the problem to $\OO$ and $\OO[1]$.
    For rank $1$ subobjects, the Hodge index theorem gives an intersection-theoretic comparison with the three $(-1)$-curves.
    For higher rank, we use the local rank reduction in Bertram's lemma and compare the rank and the degree $c_1(F).H$.
    At the rank $1$ walls, we use the stability of sheaves on the $(-1)$-curves, whose proof is deferred to \cref{section: stability of torsion sheaves}.
    We conclude \cref{section: stability} with \cref{example: irreducible detection fails} and \cref{example: maximality fails}, showing that irreducible negative curves do not suffice on an arbitrary surface and that the three detecting objects on $X$ need not include a subobject of largest slope.
    \par
    In \cref{section: stability of torsion sheaves}, the same local rank reduction and a Hodge-index calculation prove \cref{Proposition in the introduction: stability of torsions} on an arbitrary smooth projective surface.
    The numerical wall equation yields boundary points at the extreme Harder--Narasimhan slopes with $t>0$.
\end{Out}

\begin{NoCon}
We use the following notation: 
    \begin{itemize}
        \item $X$ is a smooth projective variety over $\CB$, 
        \item $D^b(X)\coloneqq D^b(\Coh(X))$ the bounded derived category of coherent sheaves on $X$, and
        \item $\Stab(X)$ the space of stability conditions associated with the numerical Grothendieck group $K_{\mathrm{num}}(X)$. 
    \end{itemize}
\end{NoCon}

\begin{Ack}
    The authors are grateful to Prof. Hajime Kaji and the second author's advisor Prof. Yasunari Nagai for their continuous encouragement and fruitful discussions.
    During the preparation of this article, the first author was partially supported by the Grant-in-Aid for
JSPS Fellows (Grant Number 22J11405), Grant-in-Aid for Research Activity Start-up (Grant Number 24K22841) and NWO (doi:10.61686/RZKLF82806), and the second author was supported by JST SPRING, Grant Number JPMJSP2128.
\end{Ack}

\begin{AI}
    LLMs were used during the revision of the previous version of the present paper, including in reworking arguments, improving the exposition, and proofreading.
\end{AI}
\section{Preliminaries on Bridgeland Stability Conditions}\label{section: preliminaries}
In this section, we recall the definitions and facts about Bridgeland stability conditions used below, with emphasis on surfaces. 
\subsection{Bridgeland Stability Conditions}\label{subsection: Bridgeland Stability Conditions}

\begin{dfn}\label[dfn]{definition: stability functions on abelian categories}
    A \emph{stability function} on an abelian category $\AC$ is a group homomorphism $Z:K(\AC)\to\CB$ such that $Z(A)\in\HB^+\coloneqq\HB\cup\RB_{<0}$ for every $0\neq A\in\AC$.
    Here, $K(\AC)$ is the Grothendieck group of $\AC$ and $\HB$ is the open upper half-plane.
\end{dfn}

\begin{dfn}[{\cite[Lemma 5.11]{macri_schmidt_2017_lectures_on_bridgeland_stability}}]\label[dfn]{definition: stability conditions on triangulated categories}
    Let $\DC$ be a triangulated category whose numerical Grothendieck group modulo torsion is a free abelian group $\Lambda$ of finite rank.
    A \emph{numerical Bridgeland stability condition} $\sigma$ on $\DC$ is a pair $(Z,\AC)$, where $\AC$ is the heart of a bounded $t$-structure on $\DC$ and $Z:\Lambda\to\CB$ induces a stability function on $\AC$ with the Harder--Narasimhan property.
    Thus, every nonzero object of $\AC$ has a finite filtration with semistable factors of strictly decreasing Bridgeland slopes.
    We also require the \emph{support property}: for a norm on $\Lambda\otimes_{\ZZ}\RB$, there is a constant $c>0$ such that $\|[A]\|\leq c|Z(A)|$ for every nonzero $\sigma$-semistable object $A$.
    Here, numerical classes of objects of $\AC$ are taken through its inclusion in $\DC$.
\end{dfn}

Denote the space of these numerical stability conditions by $\Stab(\DC)$.
Since we only consider numerical stability conditions, we simply call them stability conditions.
Bridgeland's deformation theorem gives $\Stab(\DC)$ a complex manifold structure, with the central charge as a local coordinate (\cite{bridgeland_2007_stability_conditions_on_triangulated_categories,bridgeland_2009_spaces_of_stability_conditions}).
For a smooth projective variety $X$, write $\Stab(X)\coloneqq\Stab(D^b(X))$.
For a stability condition $\sigma=(Z,\AC)$ and $0\neq E\in\AC$, the \emph{phase} $\phi(E)\in(0,1]$ is defined by $Z(E)\in\RB_{>0}\exp(i\pi\phi(E))$, and we put $\phi(E[k])\coloneqq\phi(E)+k$ for $k\in\ZZ$.

\begin{dfn}\label[dfn]{Definition: Geometric Stability Condition}
    A stability condition $\sigma$ is called \emph{geometric} 
    if all the skyscraper sheaves $\OO_{x}$ are stable of the same phase.
\end{dfn}
\begin{dfn}\label[dfn]{Definition: Bridgeland slope function}
    For $\sigma = (Z,\AC)\in\Stab(X)$ and $0\neq F\in\AC$, 
    \[
    \beta(F) \coloneqq -\frac{\Re Z(F)}{\Im Z(F)} \in (-\infty, \infty]
    \]
    is called the \emph{Bridgeland slope function}, with $\beta(F)=+\infty$ when $\Im Z(F)=0$. 
\end{dfn}
\begin{dfn}\label[dfn]{Definition: Briegeland stable object}
    For a Bridgeland stability condition $\sigma = (Z, \AC)$, a nonzero object $A\in \AC$ is $\sigma$-stable (resp. $\sigma$-semistable) if
    $\beta(B)<\beta(A)$ (resp. $\beta(B)\le\beta(A)$) for any nonzero proper subobject $B\subset A$.
    A nonzero object $F$ of $D^{b}(X)$ is $\sigma$-stable (resp. $\sigma$-semistable) if
    $F[k]\in\AC$ for some $k\in\ZZ$ and $F[k]$ is $\sigma$-stable (resp. $\sigma$-semistable) as an object of $\AC$.
\end{dfn}

\subsubsection{Stability conditions on surfaces}\label{subsubsection: Stability conditions on surfaces}

We recall the construction of \emph{divisorial} Bridgeland stability conditions.
In this construction, $N^1(X)$ and $\Amp(X)$ are understood as real numerical divisor classes and the real ample cone, respectively.

\begin{dfn}\label[dfn]{definition: divisorial stability condition}
    Let $X$ be a smooth projective surface and $(D,H)\in N^1(X)\times\Amp(X)$.
    For a nonzero torsion-free sheaf $F$, put $\mu_H(F)\coloneqq c_1(F).H/\rk(F)$.
    Let $\FC_{D,H}$ consist of the torsion-free sheaves whose nonzero subsheaves all have slope at most $D.H$, and let $\TC_{D,H}$ consist of the sheaves whose nonzero torsion-free quotients all have slope greater than $D.H$.
    The tilted heart is
    \[
    \mathcal{A}_{D,H} = \{A \in D^{b}(X)\ \mid\ H^{i}(A)= 0\ \text{for}\ i\neq -1,0,\ 
    H^{-1}(A)\in\mathcal{F}_{D,H}, H^{0}(A)\in \mathcal{T}_{D,H}
    \},
    \]
    and the central charge $Z_{D,H}:K_{\mathrm{num}}(X)\to\CB$ is defined by
    \[
        Z_{D,H}(F)=-\int_X\exp(-(D+iH))\ch(F).
    \]
\end{dfn}

\begin{prop}[{\cite[Theorem 6.10]{macri_schmidt_2017_lectures_on_bridgeland_stability}}]\label[prop]{proposition: Divisorial stability condition}
    Let $X$ be a smooth projective surface. 
    There exists a continuous embedding 
    \begin{align*}
        N^{1}(X)\times\Amp(X)&\longrightarrow \Stab(X)\\
        (D, H)&\longmapsto\sigma_{D,H} = (Z_{D,H}, \mathcal{A}_{D,H}).
    \end{align*}
\end{prop}
    We denote the image of this embedding by 
    $\divstab(X)$.

Every divisorial stability condition is geometric by \cite[Lemma 3.3]{lo_martinez_2023_geometric_stability_conditions_under_autoequivalences_and_applications_elliptic_surfaces}.

\subsubsection{Numerical walls}\label{subsubsection: Wall in spaces of BSC}
We distinguish numerical equal-slope loci from actual destabilizing walls.
The former are determined by central charges; the latter also require an exact sequence in the heart.
\begin{dfn}[Numerical walls]
    For $F,F'\in D^b(X)$, define
    \[
        \wallequation{F}{F'}\coloneqq
        -\Re Z_{D,H}(F)\Im Z_{D,H}(F')+\Re Z_{D,H}(F')\Im Z_{D,H}(F).
    \]
    We write
    \[
        \wall{F}{F'}\coloneqq
        \{\sigma_{D,H}\in\divstab(X)\mid\wallequation{F}{F'}=0\}
    \]
    for the \emph{numerical equal-slope locus}.
    When it is a proper hypersurface, we call it a numerical wall.
    The equation alone does not imply that $F$ and $F'$ lie in a common heart, or that either object destabilizes the other.
\end{dfn}

For a nonzero proper subobject $F\subset B$ in $\AC_\sigma$, we say that $F$ \emph{weakly} (resp.\ \emph{strictly}) destabilizes $B$ if $\beta_\sigma(F)\geq\beta_\sigma(B)$ (resp.\ $\beta_\sigma(F)>\beta_\sigma(B)$).
An actual destabilizing wall consists of points of a numerical wall at which such a subobject has the same slope as $B$ and $B$ is semistable.
An equation vanishing identically on a slice is a degenerate numerical locus, not a wall in that slice.

\subsection{Geometry of the del Pezzo surface}\label{subsection: the geometry of del pezzo}
Let $X$ be the blow-up of $\PPT$ at two distinct points, let $\ell$ be the pullback of the class of a line, and let $E_1,E_2$ be the exceptional curves.
The strict transform $E\coloneqq\ell-E_1-E_2$ of the line through the two points is the third $(-1)$-curve; we also write $E_3=E$.
Contracting $E$ gives a morphism $\pi:X\to Q$, where $Q\simeq\PPO\times\PPO$.
The classes $C_1\coloneqq\ell-E_1$ and $C_2\coloneqq\ell-E_2$ are the pullbacks of the two ruling classes on $Q$, and
\[
    \Pic(X)=\ZZ[C_1]\oplus\ZZ[C_2]\oplus\ZZ[E].
\]
They satisfy
\begin{align*}
    \ell=C_1+C_2-E,&\quad E_1=C_2-E,\quad E_2=C_1-E,\\
    C_1^2=C_2^2=0,\quad C_1.C_2=1,&\quad E^2=-1,\quad C_1.E=C_2.E=0.
\end{align*}
The divisor $aC_1+bC_2+cE$ is ample if and only if
\[
    a+c>0,\quad b+c>0,\quad c<0.
\]
The effective cone is generated by $E_1,E_2,E$; see, for example, \cite[Section 3]{preprint_junyan_2022_moduli_spaces_of_sheaves_on_general_blowups_of_mathbbp2}.

\subsection{Divisorial Stability Conditions on a Surface of Picard Rank Three}\label{section: Quiver region and strategy of the proof}\label{subsection: stability: settings and notations}
Let $X$ be the del Pezzo surface of Picard rank $3$ described in \cref{subsection: the geometry of del pezzo}.
For $F\in D^b(X)$, the central charge of a divisorial stability condition has the expression
\[
    Z_{D,H}(F)=-\ch_2(F)+c_1(F).D-\frac{\rk(F)}{2}(D^2-H^2)
    +i\bigl(c_1(F).H-\rk(F)D.H\bigr).
\]
We introduce the slices used to study this expression.

Let $H$ be an ample $\RB$-divisor and $G\in H^\perp$ such that $H^2=-G^2=1$.
Define
\[
    \SC_{H,G}\coloneqq\{\sigma_{sH+uG,tH}\mid s,u\in\RB,\ t>0\}.
\]
We identify this slice with the half-space of coordinates $(s,u,t)$ with $t>0$.

By the Hodge index theorem, we can choose $G_1,G_2\in H^\perp$ such that
\[
    G_1^2=G_2^2=-1,\quad G_1.G_2=0.
\]
Then $H,G_1,G_2$ form a basis of $N^1(X)_{\RB}=\Pic(X)\otimes_{\ZZ}\RB$.
We write
\[
    \SC_{H,G_1,G_2}\coloneqq\{\sigma_{sH+u_1G_1+u_2G_2,tH}\mid s,u_1,u_2\in\RB,\ t>0\}
\]
for the full divisorial family with ample class on the ray generated by $H$.

\section{Bridgeland Stability of Line Bundles}\label{section: stability}
In this section, we study the Bridgeland stability of line bundles on the del Pezzo surface $X$ of Picard rank $3$. 
Our strategy follows Arcara and Miles \cite{arcara_miles_2016_bridgeland_stability_of_line_bundles_on_surfaces,arcara_miles_2017_projectivity_of_bridgeland_moduli_spaces_on_del_pezzo_surfaces_of_picard_rank_2}, with additional wall and rank-reduction arguments for Picard rank $3$. 

First, we reduce the problem to the case of the structure sheaf.
\begin{lem}[{\cite[Lemma 3.1.]{arcara_miles_2016_bridgeland_stability_of_line_bundles_on_surfaces}}]\label[lem]{Lemma: Arcara Miles 3.1}
    Let $\sigma_{D, H}$ be a divisorial stability condition,
    $L = \OO(D')$ be a line bundle on $X$, and $F$ be a nonzero proper subobject of $L$ in $\AC_{D, H}$.
    Then, $F$ weakly (resp. strictly) destabilizes $L$ at $\sigma_{D, H}$ if and only if
    $\OO(-D')\otimes F$ weakly (resp. strictly) destabilizes $\OO$ at $\sigma_{D-D', H}$.
    The same assertion holds for a nonzero proper subobject of $L[1]$, with $\OO[1]$ in place of $\OO$.
\end{lem}
\begin{proof}
    Tensoring by $\OO(-D')$ preserves torsion sheaves and $\mu_H$-stability, and shifts every Harder--Narasimhan slope by $-D'.H$.
    It therefore identifies the torsion pairs defining $\AC_{D,H}$ and $\AC_{D-D',H}$ and gives an exact equivalence between these hearts.
    The central charges satisfy
    \[
        Z_{D,H}(F)=Z_{D-D',H}(F\otimes\OO(-D')),
    \]
    so this equivalence preserves nonzero proper subobjects and their Bridgeland slopes.
\end{proof}

Thus, it suffices to study the structure sheaf and its shift.

\begin{thm}\label{Theorem: stability of line bundle}
    Let $H$ be an ample $\RB$-divisor such that $H^2=1$, let $t>0$, and let $\sigma=\sigma_{D,tH}\in\divstab(X)$.
    Let $L$ be a line bundle on $X$ with $L\in\AC_{D,tH}$.
    Write $\beta$ for the Bridgeland slope at $\sigma$, and let $E_1,E_2,E_3$ be the $(-1)$-curves on $X$.
    Then, $L$ is not $\sigma$-stable if and only if there exists $i\in\{1,2,3\}$ such that the natural morphism $L(-E_i)\to L$ is a subobject in $\AC_{D,tH}$ and $\beta(L(-E_i))\geq\beta(L)$.
    Moreover, $L$ is not $\sigma$-semistable if and only if the same condition holds with $\beta(L(-E_i))>\beta(L)$.
\end{thm}

\begin{lem}[{\cite[Lemma 4.1]{arcara_miles_2016_bridgeland_stability_of_line_bundles_on_surfaces}}]\label[lem]{lemma: induction start with 1}
    Let $\sigma_{D,tH}$ be a divisorial stability condition, and let
    \[
    0\to F\to\OO\to Q\to0
    \]
    be an exact sequence in $\AC_{D,tH}$ with $F\neq0$.
    Then, $F$ is a torsion-free sheaf, and coherent cohomology gives an exact sequence
    \[
    0\to H^{-1}(Q)\to F\to\OO\to H^0(Q)\to0,
    \]
    where $H^{-1}(Q)\in\FC_{D,tH}$ and $H^0(Q)$ is torsion.
    Moreover, the image of $F\to\OO$ in $\Coh(X)$ is $I_Z\otimes\OO(-C)$ for an effective divisor $C$, possibly zero, and a zero-dimensional subscheme $Z\subset X$, possibly empty.
\end{lem}
\begin{proof}
    Taking coherent cohomology shows that $F$ is a sheaf in $\TC_{D,tH}$ and gives the displayed sequence.
    Its morphism $F\to\OO$ is nonzero: otherwise, $F\simeq H^{-1}(Q)$ would belong to both $\TC_{D,tH}$ and $\FC_{D,tH}$.
    Thus, its image is a rank $1$ subsheaf of $\OO$, necessarily of the asserted form, and $H^0(Q)$ is torsion.
    The torsion subsheaf of $F$ maps to zero in $\OO$, so it is contained in the torsion-free sheaf $H^{-1}(Q)$ and must vanish.
\end{proof}

We first treat rank $1$ subobjects and then reduce the general case by comparing ranks and degrees $c_1(F).H$.
For a torsion-free sheaf, let $\mu_{H,\max}$ and $\mu_{H,\min}$ denote the largest and smallest slopes of its Harder--Narasimhan factors.
We use the convention $\mu_{H,\max}(0)=-\infty$.

\begin{lem}[cf. {\cite[Remark 4.6]{arcara_miles_2016_bridgeland_stability_of_line_bundles_on_surfaces}}]\label[lem]{lemma: subobject and filtration}
    Let $H$ be an ample $\RB$-divisor, let $D.H<0$ and $t>0$, and let $F\to\OO$ be a nonzero morphism from a torsion-free sheaf.
    Set $K\coloneqq\ker(F\to\OO)$ in $\Coh(X)$.
    Then, this morphism is a subobject in $\AC_{D,tH}$ if and only if
    \[
    \mu_{H,\max}(K)\leq D.H<\mu_{H,\min}(F).
    \]
\end{lem}
\begin{proof}
    The cokernel of the morphism in $\Coh(X)$ is torsion.
    Therefore, the cone belongs to $\AC_{D,tH}$ precisely when $K\in\FC_{D,tH}$.
    Together with $F\in\TC_{D,tH}$, this is equivalent to the morphism being a subobject of $\OO$ in the heart.
    These two membership conditions are exactly the two slope inequalities.
\end{proof}

\begin{rem}\label[rem]{remark: consider only the case of s<0}
    Multiplying $H$ by $t>0$ multiplies both $\mu_H$ and $D.H$ by $t$, so the torsion pair and the heart are independent of $t$.
    In particular, if $H^2=1$ and $D=sH+U$ with $H.U=0$, then $\OO\in\AC_{D,tH}$ if and only if $s<0$.
    We work in this locus when studying subobjects of $\OO$.
    The case of $\OO[1]$ is discussed in \cref{proposition: stability of shifted sheaf}.
\end{rem}

\subsection{Rank \texorpdfstring{$1$}{1} case}\label{subsection: line bundle stability rank 1}
In this subsection, we show that if a rank $1$ subobject of $\OO$ destabilizes $\OO$, then one of $\OO(-E_i)$ also destabilizes $\OO$ at the same stability condition.

Let $H$ be an ample $\RB$-divisor such that $H^2=1$, and fix a stability condition $\sigma=\sigma_{D,tH}\in\divstab(X)$ with $D.H<0$ and $t>0$.
Throughout this subsection, we write $\beta\coloneqq \beta_\sigma$ for the Bridgeland slope at $\sigma$.
Put $\alpha\coloneqq-D.H>0$ and $V\coloneqq-(D+\alpha H)$, so that $H.V=0$.

By the Hodge index theorem, the intersection form on $H^\perp$ is negative definite.
We use the norm defined by $\|W\|^2=-W^2$ on this space.
Define the $\RB$-divisor
\[
M_\sigma\coloneqq\frac{\alpha^2+\|V\|^2+t^2}{2\alpha}H+V.
\]
Note that $M_\sigma^2>0$ and $M_\sigma.H>0$.

\begin{lem}\label[lem]{lemma: numerical criterion for rank 1 destabilization}
    Let $C$ be a nonzero effective divisor on $X$.
    Then, the following conditions are equivalent:
    \begin{enumerate}
        \item\label{item: O(-C) destabilizes O} the natural morphism $\OO(-C)\to\OO$ is a subobject in $\AC_{D,tH}$ and $\beta(\OO(-C))\geq \beta(\OO)$ (resp. $\beta(\OO(-C)) > \beta(\OO)$), and
        \item\label{item: MC le CC/2} $M_\sigma.C\leq C^2/2<0$ (resp. $M_\sigma.C<C^2/2<0$).
    \end{enumerate}
\end{lem}
\begin{proof}
    Put $h\coloneqq C.H>0$.
    The line bundle $\OO(-C)$ belongs to $\TC_{D,tH}$ if and only if $h<\alpha$.
    As $\OO$ and $\OO_C$ are in $\TC_{D, tH}$, it is also equivalent to the condition that $\OO(-C)\to \OO$ is a subobject in $\AC_{D, tH}$.
    
    When $h<\alpha$, since
    \[
    Z_{D,tH}(\OO(-C))=\frac{t^2-(D+C)^2}{2}-it(D+C).H, 
    \]
    we obtain 
    \begin{equation}\label{equation: rank 1 slope difference}
        \beta(\OO(-C))-\beta(\OO)
        =\frac{C^2/2-M_\sigma.C}{t(\alpha-h)}.
    \end{equation}
    Write $C=hH+W$ with $H.W=0$.
    A direct calculation shows that
    \begin{equation}\label{equation: rank 1 square completion}
        2\left(M_\sigma.C-\frac{C^2}{2}\right)
        =\frac{h}{\alpha}\left\|V-\frac{\alpha}{h}W\right\|^2
        +\frac{\alpha-h}{h}C^2+\frac{h}{\alpha}t^2.
    \end{equation}

    (\labelcref{item: O(-C) destabilizes O}) $\Rightarrow$ (\labelcref{item: MC le CC/2}):
    Under (\labelcref{item: O(-C) destabilizes O}), we have $h<\alpha$, and \eqref{equation: rank 1 slope difference} gives $M_\sigma.C\leq C^2/2$ (resp. $M_\sigma.C<C^2/2$).
    If $C^2\geq0$, then \eqref{equation: rank 1 square completion} gives $M_\sigma.C-C^2/2>0$, a contradiction.

    (\labelcref{item: MC le CC/2}) $\Rightarrow$ (\labelcref{item: O(-C) destabilizes O}):
    Suppose that $C^2<0$ and $M_\sigma.C\leq C^2/2$ (resp. $M_\sigma.C<C^2/2$).
    If $h\geq\alpha$, the right-hand side of \eqref{equation: rank 1 square completion} is positive, a contradiction.
    Thus, $h<\alpha$, and (\labelcref{item: O(-C) destabilizes O}) follows from \eqref{equation: rank 1 slope difference}.
\end{proof}

We next compare effective divisors of negative self-intersection with the three $(-1)$-curves.

\begin{lem}\label[lem]{lemma: comparison of effective divisors}
    Let $M$ be an $\RB$-divisor such that $M^2>0$ and $M.H>0$.
    Assume that $M.E_i\geq-1/2$ (resp. $M.E_i > -1/2$) for $i=1,2,3$.
    Then, for any effective divisor $C$ with $C^2<0$, we have
    \[
    M.C\geq\frac{C^2}{2} \quad \left(\text{resp. } M.C > \frac{C^2}{2}\right).
    \]
\end{lem}
\begin{proof}
    We argue by induction on the number of irreducible components of $C$, counted with multiplicity.
    By \cref{subsection: the geometry of del pezzo}, the effective cone of $X$ is generated by $E_1,E_2,E_3$.
    Thus, $C^2<0$ implies that $C.E_i<0$ for some $i\in\{1,2,3\}$.
    In particular, $E_i$ is a component of $C$, so $C'\coloneqq C-E_i$ is effective.
    Since $C.E_i$ is a negative integer and $E_i^2=-1$, we have $C'.E_i=C.E_i+1\leq0$.

    If $C'^2<0$, the induction hypothesis gives $M.C'\geq C'^2/2$.
    Hence,
    \[
    M.C=M.C'+M.E_i
    \geq\frac{C'^2-1}{2}
    =\frac{C^2}{2}-C'.E_i
    \geq\frac{C^2}{2}.
    \]

    If $C'^2\geq0$, then $M.C'\geq0$ by the Hodge index theorem, since $C'$ is effective and $M^2>0$, $M.H>0$.
    As $C^2$ is a negative integer, we obtain
    \[
    M.C=M.C'+M.E_i
    \geq-\frac12
    \geq\frac{C^2}{2}.
    \]
    This includes $C'=0$ and therefore also proves the initial case of the induction.
    Under the strict hypothesis, the first inequality in each display is strict.
\end{proof}

We now prove the rank $1$ case.

\begin{prop}\label[prop]{proposition: rank 1 destabilizing subobjects}
    Let $F\subset\OO$ be a nonzero proper subobject in $\AC_{D,tH}$ such that $\rk(F)=1$.
    If $\beta(F)\geq\beta(\OO)$ (resp. $\beta(F) > \beta(\OO)$), then there exists $i\in\{1,2,3\}$ such that
    \[
    \OO(-E_i)\subset\OO\text{ in }\AC_{D,tH}
    \quad\text{and}\quad
    \beta(\OO(-E_i))\geq\beta(\OO) \ (\text{resp. }\beta(\OO(-E_i))>\beta(\OO)).
    \]
\end{prop}
\begin{proof}
    By \cref{lemma: induction start with 1}, $F$ is a torsion-free sheaf.
    Since $F\to\OO$ is a nonzero morphism between torsion-free sheaves of rank $1$, it is injective in $\Coh(X)$.
    Thus, $F\simeq I_Z\otimes\OO(-C)$ for an effective divisor $C$ and a zero-dimensional subscheme $Z\subset X$.
    Put $h\coloneqq C.H$.
    The sheaves $F$ and $\OO(-C)$ are $\mu_H$-stable of the same slope $-h$.
    Since $F\in\TC_{D,tH}$, we have $h<\alpha$, and $\OO(-C)\in\TC_{D,tH}$.
    The natural morphism $\OO(-C)\to\OO$ is therefore also a subobject in $\AC_{D,tH}$.

    Denote the length of $Z$ by $\ell(Z)$.
    Since $Z_{D,tH}(F)=Z_{D,tH}(\OO(-C))+\ell(Z)$ and their imaginary parts are $t(\alpha-h)>0$, we have
    \[
    \beta(F)=\beta(\OO(-C))-\frac{\ell(Z)}{t(\alpha-h)}
    \leq\beta(\OO(-C)).
    \]
    If $C=0$, then $\ell(Z)>0$ because $F$ is a proper subobject, and this equality gives $\beta(F)<\beta(\OO)$, a contradiction.
    Thus, $C\neq0$.
    By \cref{lemma: numerical criterion for rank 1 destabilization}, we have $C^2<0$ and $M_\sigma.C\leq C^2/2$ (resp. $M_\sigma.C<C^2/2$).
    The contrapositive of the strict (resp. non-strict) assertion of \cref{lemma: comparison of effective divisors} then gives $M_\sigma.E_i\leq-1/2$ (resp. $M_\sigma.E_i<-1/2$) for some $i\in\{1,2,3\}$.
    Since $E_i^2=-1$, applying \cref{lemma: numerical criterion for rank 1 destabilization} to $E_i$ shows that $\OO(-E_i)\subset\OO$ in $\AC_{D,tH}$ and $\beta(\OO(-E_i))\geq\beta(\OO)$ (resp. $\beta(\OO(-E_i))>\beta(\OO)$), as required.
\end{proof}

\subsection{Higher rank case}\label{subsection: line bundle stability higher rank}
We first record the local rank reduction in Bertram's lemma.
We state it on an arbitrary smooth projective surface, since it is also used in \cref{section: stability of torsion sheaves}.
We will also use the stability of sheaves supported on the $(-1)$-curves when a numerical wall meets one of the rank $1$ walls.
Unless explicitly stated otherwise, Harder--Narasimhan filtrations in this subsection are taken with respect to $\mu_H$-stability, and their terms are subsheaves, not successive quotients.

\begin{lem}[Bertram's rank reduction, cf. {\cite[Lemma 4.7]{arcara_miles_2016_bridgeland_stability_of_line_bundles_on_surfaces}}]\label[lem]{lemma: Bertram's lemma for line bundle}
    Let $Y$ be a smooth projective surface, and let $H$ be an ample $\RB$-divisor on $Y$.
    Let $F\to B$ be a morphism of coherent sheaves on $Y$ with torsion cokernel, where $F$ and $K\coloneqq\ker(F\to B)$ are nonzero torsion-free sheaves.
     Put
    \[
    a\coloneqq\mu_{H,\max}(K),\quad b\coloneqq\mu_{H,\min}(F),
    \]
    and assume that $a<b$.

    Let $c\in\{a,b\}$, and let $I\subset(a,b)$ be an open interval with an endpoint $c$.
    Consider a continuous family $\{\sigma_s\}_{s\in I}$, where $\sigma_s\coloneqq\sigma_{D_s,t_sH}$ with $D_s.H=s$.
    Assume that $\sigma_s$ has a limit
    \[
      \sigma_* \coloneqq \sigma_{D_*,t_*H} = \lim_{s\to c}\sigma_s \quad\text{with}\quad t_*>0
    \]
    such that

    \begin{enumerate}[(i)]
        \item\label{item: line rank reduction family} $B\in\TC_{D_s,t_sH}$ and $\beta_{\sigma_s}(F)=\beta_{\sigma_s}(B)$ for $s\in I$, and
        \item\label{item: line rank reduction limit} $B\in\TC_{D_*,t_*H}$ and $\Im Z_{\sigma_*}(B)>0$.
    \end{enumerate}

    Let $0=K_0\subset K_1\subset\cdots\subset K_m=K$ and $0=F_0\subset F_1\subset\cdots\subset F_n=F$
    be the Harder--Narasimhan filtrations of $K$ and $F$, respectively.
    Then, the following hold:
    \begin{enumerate}
        \item\label{item: line rank reduction at lower endpoint} if $c=a$, then $F/K_1$ is a nonzero proper subobject of $B$, which strictly destabilizes $B$ at $\sigma_*$ and $\sigma_s$ for $s$ sufficiently close to $c$,
        \item\label{item: line rank reduction at upper endpoint} if $c=b$, then $F$ is not $\mu_H$-semistable, and $F_{n-1}$ is a nonzero proper subobject of $B$, which strictly destabilizes $B$ at $\sigma_s$ for $s$ sufficiently close to $c$.
    \end{enumerate}
    In both cases, the destabilizing subobject has rank strictly smaller than $\rk(F)$.
\end{lem}

\begin{proof}
      Let $\sigma_{D,tH}$ be a stability condition with $a\leq D.H<b$ and $B\in\TC_{D,tH}$.
    Then $K\in\FC_{D,tH}$ and $F\in\TC_{D,tH}$.
    Since the cokernel of $F\to B$ is torsion, the morphism is a subobject in $\AC_{D,tH}$.
    This applies to $\sigma_s$ for $s\in I$, and to $\sigma_*$ if $c=a$.

    (\labelcref{item: line rank reduction at lower endpoint}): Suppose that $c=a$, so that $D_*.H=a$.
    Then the sheaf $K_1$ is $\mu_H$-semistable of slope $a$.
    Thus, $K_1[1]$ belongs to $\AC_{D_*,t_*H}$ and
    \[
      Z_{\sigma_*}(K_1) = \Re Z_{\sigma_*}(K_1) + it_*\rk(K_1)(a-D_*.H) = \Re Z_{\sigma_*}(K_1) \in \RB_{>0}.
    \]
    The sheaf $F/K_1$ belongs to $\TC_{D_*,t_*H}$ because $\TC_{D_*,t_*H}$ is closed under taking a quotient, and $K/K_1$ belongs to $\FC_{D_*,t_*H}$ because its Harder--Narasimhan filtration is given by $0 \subset K_2/K_1 \subset \cdots \subset K/K_1$ and $\mu_{H,\max}(K/K_1)<a$.
    Hence, $F/K_1\to B$ is a subobject in $\AC_{D_*,t_*H}$, since its cone $Q$ satisfies
    $H^{-1}(Q)=K/K_1\in\FC_{D_*,t_*H}$ and $H^0(Q)=\operatorname{coker}(F\to B)$, which is torsion
    and hence lies in $\TC_{D_*,t_*H}$.

    Since $\Im Z_{\sigma_{*}}(F)>0$, $Z_{\sigma_*}(K_1) \in \RB_{>0}$ and 
    \[
      Z_{\sigma_*}(F/K_1)=Z_{\sigma_*}(F)-Z_{\sigma_*}(K_1),
    \]
    we obtain
    \[
    \beta_{\sigma_*}(F/K_1)>\beta_{\sigma_*}(F)=\beta_{\sigma_*}(B).
    \]
    Here, $\beta_{\sigma_*}(F)=\beta_{\sigma_*}(B)$ follows from (\labelcref{item: line rank reduction family}) by taking the limit.
    This also shows that the subobject $F/K_1$ is nonzero and proper. Its rank is smaller than $\rk(F)$ because $K_1$ has positive rank.
    
    For $s\in I$ sufficiently close to $c$, the sheaf $F/K_1$ remains in $\TC_{D_s,t_sH}$ as a quotient of $F$, and all Harder--Narasimhan slopes of $K/K_1$ are smaller than $a$.
    Thus, the same heart membership conditions hold, and the strict inequality persists by continuity.

    (\labelcref{item: line rank reduction at upper endpoint}): Suppose that $c=b$, so that $D_*.H=b$, and put $T\coloneqq F/F_{n-1}$.
    Then $T$ is $\mu_H$-semistable of slope $b$, and similarly $Z_{\sigma_*}(T)\in\RB_{>0}$.
    Suppose that $F$ is $\mu_H$-semistable. Then $Z_{\sigma_*}(F)\in\RB_{>0}$.
    Taking the limit of (\labelcref{item: line rank reduction family}), we have 
    \[
      \Re Z_{\sigma_*}(F)\Im Z_{\sigma_*}(B) - \Re Z_{\sigma_*}(B)\Im Z_{\sigma_*}(F) = 0,
    \]
    which induces $\Im Z_{\sigma_*}(B)=0$.
    This contradicts (\labelcref{item: line rank reduction limit}).
    Thus, $F_{n-1}\neq0$.

    For $s\in I$ sufficiently close to $b$, the sequence
    \[
    0\to F_{n-1}\to F\to T\to0
    \]
    is exact in the heart, and $\beta_{\sigma_s}(T)\to-\infty$ as $s\to b$.
    On the other hand, $\beta_{\sigma_s}(F)=\beta_{\sigma_s}(B)$ has a finite limit.
    The seesaw property therefore gives
    \[
    \beta_{\sigma_s}(F_{n-1})>\beta_{\sigma_s}(F)=\beta_{\sigma_s}(B)
    \]
    near $b$.
    The composite $F_{n-1}\to F\to B$ is a subobject, and the strict inequality again makes it proper. Its rank is smaller than $\rk(F)$ because $T$ has positive rank.

\end{proof}

\begin{rem}\label[rem]{remark: single point at lower endpoint}
    In the case $c=a$, the assertion at $\sigma_*$ uses the family only to obtain
    $\beta_{\sigma_*}(F)=\beta_{\sigma_*}(B)$.
    Thus, if $\sigma=\sigma_{D,tH}$ is a stability condition with $D.H=a$ at which
    $F\to B$ is a subobject in the heart and $\beta_\sigma(F)=\beta_\sigma(B)$,
    the same argument shows that $F/K_1$ is a nonzero proper subobject of $B$ of smaller rank which strictly destabilizes $B$ at $\sigma$.
\end{rem}

From now on, let $H$ be an ample $\RB$-divisor on $X$ such that $H^2=1$.
The following statement will be proved in \cref{section: stability of torsion sheaves} as a special case of \cref{proposition: stability of torsion sheaf}.
\begin{lem}\label[lem]{lemma: stability on a negative curve}
    Let $E_i$ be one of the $(-1)$-curves on $X$.
    Then, $\OO_{E_i}(k)$ is stable for every divisorial stability condition and every $k\in\ZZ$.
\end{lem}

For $D.H<0$ and $\sigma = \sigma_{D,tH}$, retain the notation
\begin{align}
  -D &=\alpha H+V, \nonumber\\ 
  M_\sigma &=\frac{\alpha^2+\|V\|^2+t^2}{2\alpha}H+V = -D-t\beta_\sigma(\OO)H \label{equation: M sigma via slope}
\end{align}
from \cref{subsection: line bundle stability rank 1}.
Note that the last equality follows from $D^2= \alpha^2 - \|V\|^2$ and $(t^2-D^2)/2\alpha = -t\beta_{\sigma}(\OO)$.
Define
\begin{align*}
  \UC_H &\coloneqq \left\{\sigma_{D,tH}\ \middle|\ D.H<0,\ t>0,\ M_\sigma.E_i>-1/2\ (i=1,2,3)\right\} \subset \SC_{H,G_1,G_2} , \\ 
  \overline{\UC}_H &\coloneqq \left\{\sigma_{D,tH}\ \middle|\ D.H<0,\ t>0,\ M_\sigma.E_i \geq -1/2\ (i=1,2,3)\right\} \subset \SC_{H,G_1,G_2}.  
\end{align*}
$\overline{\UC}_H$ is the closure of $\UC_H$ in the locus $D.H<0$, $t>0$.
Indeed, if $\sigma_{D,tH} \in \overline{\UC}_H$, then the family $\{\sigma_{D,(t+\varepsilon)H}\}_{\varepsilon >0}$ of elements of $\UC_H$ converges to $\sigma_{D,tH}$ as $\varepsilon \to 0$.
By \cref{lemma: numerical criterion for rank 1 destabilization,proposition: rank 1 destabilizing subobjects}, $\UC_H$ is precisely the locus where no nonzero proper rank $1$ subobject weakly destabilizes $\OO$.
We define $J(\sigma)$ to be 
\[
  J(\sigma) = \{i \in \{1,2,3\} \mid M_\sigma.E_i = -1/2\}.
\]

\begin{lem}\label[lem]{lemma: reduction at a rank 1 wall}
    Let $\sigma =\sigma_{D,tH}\in\overline{\UC}_H\setminus\UC_H$, and let $F\subset\OO$ be a nonzero proper subobject of rank $r>1$ with $\beta_\sigma(F)\geq\beta_\sigma(\OO)$.
    Then, there exist $\sigma'\in\UC_H$ and a nonzero proper subobject $F'\subset\OO$ at $\sigma'$ such that $\beta_{\sigma'}(F')\geq\beta_{\sigma'}(\OO)$ and either
    \[
    \rk(F')<r,
    \quad\text{or}\quad
    \rk(F')=r\text{ and }c_1(F').H\geq c_1(F).H+(r-1)\min_{1\leq i\leq3}(E_i.H).
    \]
\end{lem}

\begin{proof}
    Step 1: Choose $i \in J(\sigma)$.
    By \cref{lemma: numerical criterion for rank 1 destabilization}, the sequence
    \[
    0\to\OO(-E_i)\to\OO\to\OO_{E_i}\to0
    \]
    is exact in $\AC_{D,tH}$, and its three terms have the same Bridgeland slope.
    We denote by $\varphi$ the composite $F\to\OO\to\OO_{E_i}$.
    Let $A = \ker(\varphi)$ and $B = \operatorname{im}(\varphi)$ in $\AC_{D,tH}$.
    Then, $A\subset\OO(-E_i)$, $B\subset\OO_{E_i}$ and
    \[
        0\to A\to F\to B\to0
    \]
    is exact in $\AC_{D,tH}$.
    Any nonzero $A$ is a torsion-free sheaf by \cref{lemma: induction start with 1} since $A\subset\OO(-E_i)\subset\OO$.
    By \cref{lemma: stability on a negative curve}, if $B\neq0$, then
    \[
    \beta(B)\leq\beta(\OO_{E_i})=\beta(\OO),
    \]
    with strict inequality unless $B=\OO_{E_i}$.
    It follows that $A\neq0$. 
    Indeed, if $A=0$, then $F = B \neq \OO_{E_i}$, but $\beta_{\sigma}(F) \geq \beta_{\sigma}(\OO) = \beta_{\sigma}(\OO_{E_i})$, which is a contradiction. 
    Moreover, $\beta(A)\geq\beta(\OO)$, which follows from $\beta_{\sigma}(F) \geq \beta_{\sigma}(\OO) \geq \beta_{\sigma}(B)$ and the seesaw property.

    Step 2: Set $\tau\coloneqq\sigma_{D+E_i,tH}$ and $F_1\coloneqq A(E_i)$.
    Tensoring with $\OO(E_i)$ gives a subobject $F_1\subset\OO$ at $\tau$ with $\beta_\tau(F_1)\geq\beta_\tau(\OO)$.
    The stability condition $\tau$ satisfies $(D+E_i).H<0$, since $\OO(-E_i)$ belongs to the original heart $\AC_{D,tH}$ (\cref{lemma: subobject and filtration}).
    From $\beta_\tau(\OO)=\beta_\sigma(\OO(-E_i))=\beta_\sigma(\OO)$ and \eqref{equation: M sigma via slope}, we have
    \begin{equation}\label{equation: divisor at a rank 1 wall}
        M_\tau=M_\sigma-E_i.
    \end{equation}
    Hence, $M_\tau.E_i=1/2$.
    If $E_i.E_j=0$, then $M_\tau.E_j=M_\sigma.E_j$.
    If $E_i.E_j=1$, then the divisor $N\coloneqq E_i+E_j$ is nonzero and nef,
    since the effective cone is generated by $E_1$, $E_2$ and $E_3$.
    Write $N=nH+W$ with $W\in H^\perp$. 
    Since $N^2=E_i^2+2E_i.E_j+E_j^2=0$, we have $n=N.H=\|W\|$.
    Since minus the intersection product is an inner product on $H^\perp$ by the Hodge index theorem,
    the Cauchy--Schwarz inequality gives
    \[
      M_\sigma.N\geq pn-\|V\|\|W\|=n(p-\|V\|)>0,
    \]
    where $p=M_\sigma.H$.
    Consequently, using $M_\sigma.E_i=-1/2$ and \eqref{equation: divisor at a rank 1 wall}, we have
    \[
    M_\sigma.E_j>1/2,\quad M_\tau.E_j>-1/2.
    \]
    Thus, $\tau\in\overline{\UC}_H$, and $J(\tau) = J(\sigma) \setminus \{i\}$. 

    Step 3: The object $B$, which is a subobject of $\OO_{E_i}$, is a coherent sheaf of nonnegative rank.
    If $\rk(B)>0$, then $\rk(F_1)<r$ and $B\neq\OO_{E_i}$, so $\beta_\tau(F_1)>\beta_\tau(\OO)$.
    Put $F'\coloneqq F_1$ and $\sigma'\coloneqq \sigma_{D+E_i, t'H}$ ($t' > t$).
    If $t'-t$ is sufficiently small, then $\sigma' \in \UC_H$ and the strict inequality $\beta_{\sigma'}(F')>\beta_{\sigma'}(\OO)$ persists. 
    This gives the first conclusion.

    If $\rk(B)=0$, then $\rk(F_1)=r$.
    We denote by $Q$ the quotient of the inclusion $B\subset \OO_{E_i}$ in $\AC_{D,tH}$.
    Since $H^{-1}(Q)$ is torsion-free of rank $0$, it is the zero sheaf.
    Thus, if $B\neq0$, then it is isomorphic to $\OO_{E_i}(-k)$ for $k\geq 0$. 
    In particular, $c_1(B)=E_i$.
    Accordingly, we have
    \[
    c_1(F_1)=c_1(F) + 
    \begin{cases}
      rE_i & B=0, \\
      (r-1)E_i & B\neq0,
    \end{cases}
    \]
    Thus, $c_1(F_1).H\geq c_1(F).H+(r-1)E_i.H$.
    Since $r>1$, $F_1$ is a nonzero proper subobject of $\OO$ in $\AC_{D+E_i,tH}$.

    If $\tau \in \UC_H$, the pair $(\tau, F_1)$ is the desired pair.
    If $\tau$ is not in $\UC_H$, repeat the construction at an index $j\in J(\tau)$ where equality still holds.
    Each step removes one such index and adds none.
    Either the rank decreases, with a strict destabilizing inequality as above, or it stays $r$ and the degree increases by at least $(r-1)\min_i(E_i.H)$ at each step.
    After finitely many steps, we obtain the required pair $(\sigma',F')$.
\end{proof}

\begin{prop}\label[prop]{proposition: maximal destabilize}
    Let $H$ be an ample $\RB$-divisor such that $H^2=1$, let $D.H<0$ and $t>0$, and put $\sigma=\sigma_{D,tH}$.
    Let $F\subset\OO$ be a nonzero proper subobject in $\AC_{D,tH}$.
    If $\beta_\sigma(F)\geq\beta_\sigma(\OO)$, then there exists $i\in\{1,2,3\}$ such that the natural morphism $\OO(-E_i)\to\OO$ is a subobject in $\AC_{D,tH}$ and
    \[
    \beta_\sigma(\OO(-E_i))\geq\beta_\sigma(\OO).
    \]
    Moreover, if $\beta_\sigma(F)>\beta_\sigma(\OO)$, then $i$ can be chosen so that $\beta_\sigma(\OO(-E_i))>\beta_\sigma(\OO)$.
\end{prop}

\begin{proof}
     We first show that $\OO$ is stable at every point of $\UC_H$.
    
    Suppose that $\OO$ is not stable for some stability condition in $\UC_H$.
    We define $\VC_H$ as follows: 
    \[
     \VC_H = \left\{(\sigma, F) \ \middle|\ 
     \begin{array}{l}
       \sigma \in \UC_H,\ F \subset \OO \text{ is a nonzero proper subobject in }\AC_{D,tH}\\
       \text{with } \beta_\sigma(F) \geq \beta_\sigma(\OO)
     \end{array}\right\}.
    \]
    Among all pairs in $\VC_H$, choose a pair $(\sigma, F)$ such that $r\coloneqq\rk(F)$ is least.
    The rank $1$ case gives $r>1$ from the definition of $\UC_H$.
    For any pair $(\sigma,F)$ in $\VC_H$ of rank $r$, set $K\coloneqq\ker(F\to\OO)$ and $N\coloneqq\operatorname{im}(F\to\OO)$ in $\Coh(X)$.
    By \cref{lemma: induction start with 1}, $N\cong I_Z\otimes\OO(-C)$ for an effective divisor $C$, and $K$ is torsion-free of rank $r-1$.
    Since $K\in\FC_{D,tH}$ and $D.H<0$, we have
    \[
    c_1(F).H=c_1(K).H-C.H\leq(r-1)D.H-C.H<0.
    \]
    Put
    \[
      d\coloneqq\sup\{\, c_1(F).H \mid (\sigma,F)\in\VC_H,\ \rk(F)=r \,\},\qquad \delta\coloneqq(r-1)\min_i(E_i.H)>0.
    \]
    Since $-\infty<d\leq0$, there is a pair $(\sigma,F)\in\VC_H$ with $\rk(F)=r$ and $c_1(F).H>d-\delta$.

    Put $s_0\coloneqq D.H$, and choose $G\in H^\perp$ with $G^2=-1$ such that $D=s_0H+u_0G$.
    Write
    \[
    h\coloneqq c_1(F).H<0,\quad g\coloneqq-c_1(F).G,\quad c\coloneqq\ch_2(F).
    \]
    The wall equation $\wallequation{F}{\OO}(s,u,t)= 0$ for $F$ and $\OO$ in $\SC_{H,G}$ is written as 
    \begin{equation}\label{equation: higher rank wall at positive volume}
        h(s^2+u^2+t^2)-2gsu-2cs=0.
    \end{equation}
    We have
    \[
    \beta(F)-\beta(\OO)=\frac{\wallequation{F}{\OO}}{\Im Z(F)\Im Z(\OO)},
    \]
    and $2\wallequation{F}{\OO}/t$ is the left hand side of \eqref{equation: higher rank wall at positive volume}.
    If $F\subset\OO$ in $\AC_{D,tH}$, the denominator $\Im Z(F)\Im Z(\OO)$ is positive, and hence the left hand side of \eqref{equation: higher rank wall at positive volume} has the same sign as $\beta(F)-\beta(\OO)$.
    Since $h<0$, increasing $t$ with $D$ fixed reaches this wall at some positive value $t_0$.
    Increasing $t$ does not change the heart, that is, $\AC_{D,t_0H}= \AC_{D,tH}$, and the stability condition $\sigma_{D, t_0H}$ is still in $\UC_H$ because $M_\sigma$ increases by a positive multiple of $H$.
    Hence, we may replace our point by this wall point $(s_0,u_0,t_0)$. 

    Put $a\coloneqq\mu_{H,\max}(K)$ and $b\coloneqq\mu_{H,\min}(F)$.
    Then, $a\leq s_0<b<0$, where the last inequality follows from $b\leq h/r<0$.
    If $s_0=a$, \cref{remark: single point at lower endpoint} immediately gives a strictly destabilizing subobject of smaller rank in $\UC_H$, contradicting the minimality of $r$. 
    We may thus assume that $a<s_0<b$.

    Completing the square in \eqref{equation: higher rank wall at positive volume}, set
    \[
    R_F(s)\coloneqq\frac{2c}{h}s+\frac{g^2-h^2}{h^2}s^2.
    \]
    Then, the wall is given by
    \[
    \left(u-\frac{g}{h}s\right)^2+t^2=R_F(s).
    \]
    On $s<0$, the function $R_F(s)/(-s)$ is linear in $s$ and is positive at $s_0\in(a,b)$.
    Hence, it is positive at some $s_1\in\{a,b\}$ and throughout the closed interval between $s_0$ and $s_1$.
    On this interval, $R_F>0$, so we may define
    \[
    \begin{aligned}
        u(s)&=\frac{g}{h}s+
        \left(u_0-\frac{g}{h}s_0\right)\sqrt{\frac{R_F(s)}{R_F(s_0)}},\\
        t(s)&=t_0\sqrt{\frac{R_F(s)}{R_F(s_0)}}.
    \end{aligned}
    \]
    This is a continuous path on the numerical wall with $t(s)>0$.
    For any $s \neq s_1$ between $s_0$ and $s_1$, $F\subset\OO$ remains a subobject in the heart and $\beta(F)=\beta(\OO)$.

    If the path meets $\overline{\UC}_H\setminus\UC_H$ before reaching $s_1$, apply \cref{lemma: reduction at a rank 1 wall} at such a point.
    It produces a weakly destabilizing subobject $F'$ at a point of $\UC_H$ with either rank smaller than $r$, or rank $r$ and $c_1(F').H\geq c_1(F).H+\delta>d$.
    The first case contradicts the minimality of $r$, and the second contradicts the definition of $d$.
    Otherwise, the path stays in $\UC_H$ until it reaches $s_1$.
    By \cref{lemma: Bertram's lemma for line bundle}, a subobject of smaller rank than $r$ strictly destabilizes $\OO$ for $s$ sufficiently close to $s_1$.
    The corresponding points $(s,u(s),t(s))$ of the path still belong to $\UC_H$, again a contradiction.
    Thus, $\OO$ is stable throughout $\UC_H$ for every ample $\RB$-divisor $H$.
    By \cref{lemma: numerical criterion for rank 1 destabilization}, this proves the first assertion.
    
    Finally, let $\sigma\in\overline{\UC}_H$.
    Increasing $t$ slightly places $\sigma$ in $\UC_H$, where $\OO$ is stable.
    The heart is unchanged.
    By \cref{lemma: induction start with 1}, every nonzero subobject of $\OO$ is a torsion-free sheaf in $\TC_{D,tH}$, so its central charge has positive imaginary part and its slope is a finite number.
    Thus, for every fixed nonzero proper subobject $F\subset\OO$, continuity of the slope gives
    \[
    \beta_\sigma(F)\leq\beta_\sigma(\OO)
    \]
    on taking the limit.
    Hence, $\OO$ is semistable at every point of $\overline{\UC}_H$.
    By contraposition, if $\beta_\sigma(F)>\beta_\sigma(\OO)$ for some nonzero proper subobject $F\subset\OO$ at $\sigma=\sigma_{D,tH}$, then $\sigma\notin\overline{\UC}_H$, that is, $M_\sigma.E_i<-1/2$ for some $i$.
    The strict form of \cref{lemma: numerical criterion for rank 1 destabilization} gives the last assertion.
\end{proof}

\begin{proof}[\proofname\ of \cref{Theorem: stability of line bundle}]
       Write $L=\OO(D')$.
    Since $L\in\AC_{D,tH}$, we have $(D-D').H<0$.
    Tensoring with $L^{-1}$ gives an equivalence $\AC_{D,tH} \cong \AC_{D-D',tH}$ and satisfies
    \[
    Z_{D,tH}(F)=Z_{D-D',tH}(F\otimes L^{-1}),
    \]
    as in \cref{Lemma: Arcara Miles 3.1}.
    It therefore preserves nonzero proper subobjects and slopes.
    Applying \cref{proposition: maximal destabilize} to $\OO$ at $\sigma_{D-D',tH}$ gives the required subobject $L(-E_i)\subset L$.
    Conversely, if $L(-E_i)\to L$ is a subobject in $\AC_{D,tH}$ with $\beta(L(-E_i))\geq\beta(L)$, then $L(-E_i)$ is a nonzero proper subobject of $L$ whose slope is at least that of $L$, and hence $L$ is not $\sigma$-stable by definition.
    Similarly, if $\beta(L(-E_i))>\beta(L)$, then $L$ is not $\sigma$-semistable.
\end{proof}

We next treat the shifted structure sheaf, using the duality in \cite[Section 6]{arcara_miles_2016_bridgeland_stability_of_line_bundles_on_surfaces}.

\begin{prop}\label[prop]{proposition: stability of shifted sheaf}
    Let $H$ be an ample $\RB$-divisor such that $H^2=1$, let $t>0$, and let $\sigma=\sigma_{D,tH}$ with $D.H\geq0$.
    If $D.H=0$, then $\OO[1]$ is $\sigma$-stable.
    If $D.H>0$, then $\OO[1]$ is not $\sigma$-stable if and only if there exists $i\in\{1,2,3\}$ such that the natural morphism
    \[
        \OO(E_i)|_{E_i}\to\OO[1]
    \]
    is a subobject in $\AC_{D,tH}$ and $\beta_\sigma(\OO(E_i)|_{E_i})\geq\beta_\sigma(\OO[1])$.
    Moreover, $\OO[1]$ is not $\sigma$-semistable if and only if the same condition holds with a strict inequality.
\end{prop}

\begin{proof}
  First, suppose that $D.H=0$.
    Then, $\OO[1]$ has no nonzero proper subobjects in $\AC_{D,tH}$, and hence it is $\sigma$-stable by \cite[Lemma 6.1]{arcara_miles_2016_bridgeland_stability_of_line_bundles_on_surfaces} (where $s=D.H$ in their notation).
   
    Now, suppose that $D.H>0$, and put $\sigma'\coloneqq\sigma_{-D,tH}$.
    By \cite[Lemma 6.2]{arcara_miles_2016_bridgeland_stability_of_line_bundles_on_surfaces}, $\OO[1]$ is $\sigma$-stable (resp. $\sigma$-semistable) if and only if $\OO$ is $\sigma'$-stable (resp. $\sigma'$-semistable).
    If stability fails, \cref{proposition: maximal destabilize} gives a natural subobject $\OO(-E_i)\subset\OO$ in $\AC_{-D,tH}$ with
    \[
        \beta_{\sigma'}(\OO(-E_i))\geq\beta_{\sigma'}(\OO),
    \]
    and the inequality can be chosen strict if semistability fails.
    The condition $\OO(-E_i) \in \TC_{-D,tH}$ implies $E_i.H<D.H$, so $\OO,\OO(E_i)\in\FC_{D,tH}$.
    Hence, the usual exact sequence of sheaves gives an exact sequence
    \begin{equation}\label{equation: shifted structure sheaf sequence}
        0\to\OO(E_i)|_{E_i}\to\OO[1]\to\OO(E_i)[1]\to0
    \end{equation}
    in $\AC_{D,tH}$.
    All three terms have positive imaginary part because $0<E_i.H<D.H$.
    For any divisor $C$, a direct computation gives
    \[
        Z_{D,tH}(\OO(C)[1])=-\overline{Z_{-D,tH}(\OO(-C))},
    \]
    and hence $\beta_\sigma(\OO(C)[1])=-\beta_{\sigma'}(\OO(-C))$.
    Applying this to $C=0$ and $C=E_i$ gives
    \[
        \beta_\sigma(\OO(E_i)[1])\leq\beta_\sigma(\OO[1]),
    \]
    with a strict inequality in the strict case.
    The seesaw property in \eqref{equation: shifted structure sheaf sequence} gives the required inequality for $\OO(E_i)|_{E_i}$.

    Conversely, if $\OO(E_i)|_{E_i}\to\OO[1]$ is a subobject in $\AC_{D,tH}$ with $\beta_\sigma(\OO(E_i)|_{E_i})\geq\beta_\sigma(\OO[1])$, then $\OO(E_i)|_{E_i}$ is a nonzero proper subobject of $\OO[1]$ whose slope is at least that of $\OO[1]$, and hence $\OO[1]$ is not $\sigma$-stable by definition.
    Similarly, if $\beta_\sigma(\OO(E_i)|_{E_i})>\beta_\sigma(\OO[1])$, then $\OO[1]$ is not $\sigma$-semistable.
\end{proof}

\begin{rem}
    By \cref{Lemma: Arcara Miles 3.1}, the same result holds for $L[1]$, where $L=\OO(D')$, after replacing $D$ by $D-D'$ and $\OO(E_i)|_{E_i}$ by $L(E_i)|_{E_i}$.
    In particular, $L[1]$ is stable when $(D-D').H=0$.
\end{rem}

\begin{eg}
    Let
    \[
        H=\frac{4C_1+2C_2-E}{\sqrt{15}},\quad
        G_1=-C_1+\frac12C_2,\quad
        G_2=\frac{-C_1-\frac12C_2+4E}{\sqrt{15}}.
    \]
    Then, $H^2=-G_1^2=-G_2^2=1$ and $H.G_1=H.G_2=G_1.G_2=0$, as required in \cref{subsection: stability: settings and notations}.
    To illustrate the numerical equations, identify the formal boundary $t=0$ with $\SC^0_{H,G_1,G_2}\coloneqq\RB^3$, with coordinates $(s,u_1,u_2)$, and put
    \[
        \wallpolalim{F}{F'}{H}{G_1}{G_2}
        \coloneqq\left\{(s,u_1,u_2)\in\SC^0_{H,G_1,G_2}\ \middle|\
        \left.\frac{1}{t}\wallequation{F}{F'}\right|_{t=0}=0\right\}.
    \]
    Here, $t^{-1}\wallequation{F}{F'}$ is a polynomial in $(s,u_1,u_2,t)$.
    The formal boundary is not a space of stability conditions.
    The numerical loci $\wallpolalim{\OO(-E_i)}{\OO}{H}{G_1}{G_2}$ in $\SC^0_{H,G_1,G_2}$ are given by
    \begin{align*}
        \wallpolalim{\OO(-E_1)}{\OO}{H}{G_1}{G_2}:\ & -\frac{3}{\sqrt{15}}(s^2+u_1^2+u_2^2)+2su_1-\frac{6}{\sqrt{15}}su_2+s=0, \\
        \wallpolalim{\OO(-E_2)}{\OO}{H}{G_1}{G_2}:\ & -\frac{1}{\sqrt{15}}(s^2+u_1^2+u_2^2)-su_1-\frac{7}{\sqrt{15}}su_2+s=0, \\
        \wallpolalim{\OO(-E_3)}{\OO}{H}{G_1}{G_2}:\ & -\frac{1}{\sqrt{15}}(s^2+u_1^2+u_2^2)+\frac{8}{\sqrt{15}}su_2+s=0.
    \end{align*}
    
    At $t>0$, the corresponding numerical wall equations are obtained by replacing $s^2+u_1^2+u_2^2$ with $s^2+u_1^2+u_2^2+t^2$.
    These equations alone do not describe actual destabilizing walls.
    Indeed, in the locus $D.H<0$, the part of the numerical locus of $\OO(-E_i)$ lying on the boundary of the region where $\OO$ is stable consists of the points with
    \[
        M_\sigma.E_i=-1/2,\quad M_\sigma.E_j\geq-1/2\quad(j\neq i),
    \]
    by \cref{lemma: numerical criterion for rank 1 destabilization,proposition: maximal destabilize}.

    \Cref{figure: maximal walls} shows the part of these numerical loci with $s<0$, which is the condition for $\OO\in\AC_{D,tH}$ when $t>0$.
    It is drawn directly from the displayed equations as follows.
    Put $e_i\coloneqq E_i.H$ and $b_{ij}\coloneqq E_i.G_j$.
    For each fixed $s\leq-e_i$, the intersection of the locus of $\OO(-E_i)$ with the plane $\{s=\text{const}\}$ is the circle
    \[
        (u_1,u_2)=-\frac{s}{e_i}(b_{i1},b_{i2})
        +\frac{\sqrt{s(s+e_i)}}{e_i}(\cos\theta,\sin\theta),\quad 0\leq\theta\leq2\pi,
    \]
    and the locus has no points with $-e_i<s<0$.
    The figure uses $-2\leq s\leq-e_i$ for each curve and the coordinate order $(u_1,u_2,s)$. It does not represent actual walls at $t>0$.

    \begin{figure}[ht]
        \centering
        \includegraphics[width=0.82\linewidth]{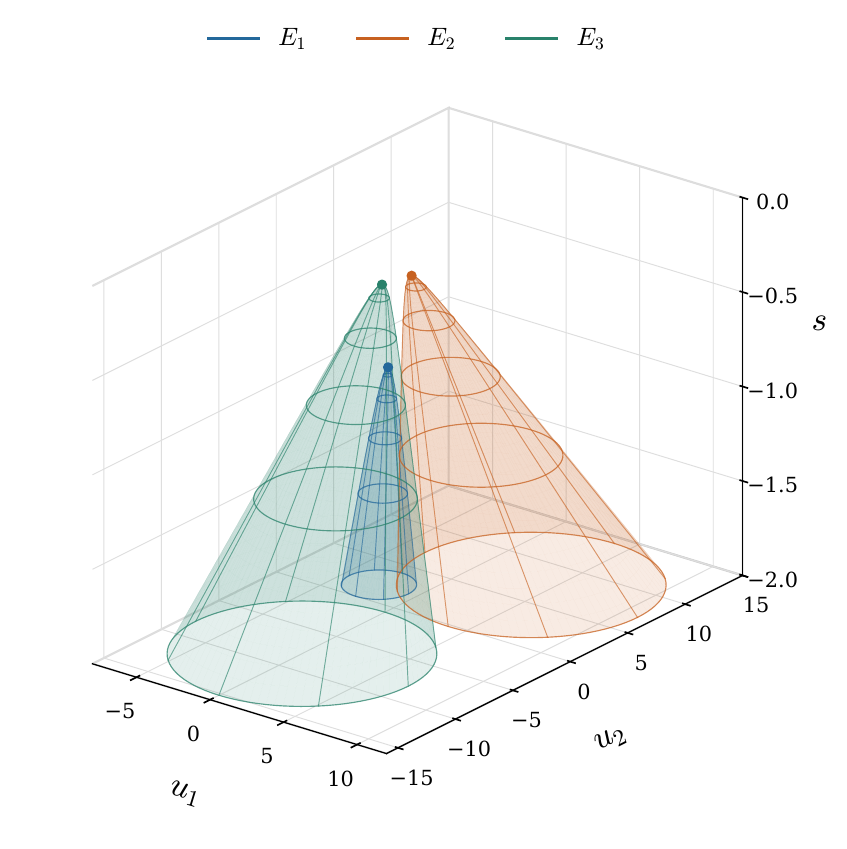}
        \caption{Numerical loci $\wallpolalim{\OO(-E_i)}{\OO}{H}{G_1}{G_2}$ at $t=0$, parametrized in the coordinates $(u_1,u_2,s)$}
        \label{figure: maximal walls}
    \end{figure}
\end{eg}

\begin{eg}\label[eg]{example: irreducible detection fails}
    The restriction to irreducible negative curves in \cref{Theorem: stability of line bundle} does not extend to arbitrary smooth projective surfaces, even for instability detection alone.
    Let $\pi:Y\to T$ be the minimal resolution of the cubic surface
    \[
        T=\{x_0x_1x_2=x_3^3\}\subset\PP^3.
    \]
    Its three singular points have local equation $uv=w^3$ and are of type $A_2$.
    Let $T_1,T_2$ be the two exceptional curves over one such point, and let $P$ be the pullback of a hyperplane section.
    Then, $P$ is nef, $P^2=3$, $P.T_1=P.T_2=0$, and
    \[
        T_1^2=T_2^2=-2,\quad T_1.T_2=1.
    \]
    Put $C\coloneqq T_1+T_2$ and $M\coloneqq P+3C/4$.
    We have
    \[
        C^2=-2,\quad M^2=15/8,\quad M.C=-3/2,\quad M.T_i=-3/4\quad(i=1,2).
    \]
    Choose an ample $\RB$-divisor $H$ on $Y$ with $H^2=1$, and set $D\coloneqq-M$, $t\coloneqq\sqrt{15/8}$ and $\sigma\coloneqq\sigma_{D,tH}$.
    We write $Z\coloneqq Z_{D,tH}$ and $\beta\coloneqq\beta_\sigma$ in this example.
    Both $M$ and $M-C$ lie in the positive cone containing the ample classes: their squares are $15/8$ and $23/8$, respectively, and their intersections with $P$ are positive.
    Consequently, $M.H>0$ and $(M-C).H>0$ by the Hodge index theorem.
    Thus, $\OO_Y(-C)$ and $\OO_Y$ belong to $\TC_{D,tH}$, and the usual exact sequence
    \[
        0\to\OO_Y(-C)\to\OO_Y\to\OO_C\to0
    \]
    is exact in $\AC_{D,tH}$.
    Direct calculation gives
    \[
        Z(\OO_Y)=itM.H,\quad
        Z(\OO_Y(-C))=-1/2+it(M-C).H.
    \]
    Hence, $\beta(\OO_Y(-C))>0=\beta(\OO_Y)$, so $\OO_Y$ is not $\sigma$-semistable.

    For an irreducible curve $\Gamma\neq T_1,T_2$, nefness of $P$ and nonnegative intersections of distinct irreducible curves give $M.\Gamma\geq0$.
    If $\Gamma^2<0$, then
    \[
        \Re Z(\OO_Y(-\Gamma))=M.\Gamma-\Gamma^2/2>0.
    \]
    For $\Gamma=T_1,T_2$, we instead compute $\Re Z(\OO_Y(-T_i))=M.T_i-T_i^2/2=-3/4+1=1/4>0$.
    Whenever $\OO_Y(-\Gamma)$ is a subobject in $\AC_{D,tH}$, its imaginary part is positive, so $\beta(\OO_Y(-\Gamma))<0=\beta(\OO_Y)$.
    Therefore, for no irreducible negative curve $\Gamma$ does the subobject $\OO_Y(-\Gamma)\subset\OO_Y$ weakly destabilize $\OO_Y$, whereas $\OO_Y(-C)\subset\OO_Y$ for the reducible negative divisor $C$ strictly destabilizes it.
\end{eg}

\begin{eg}\label[eg]{example: maximality fails}
    On the del Pezzo surface $X$ of Picard rank $3$, the three detecting objects in \cref{Theorem: stability of line bundle} need not include a subobject of largest slope.
    Put
    \[
        H\coloneqq\frac{2C_1+2C_2-E}{\sqrt7},\quad D\coloneqq-C_1-3C_2+3E,\quad t\coloneqq1,
    \]
    and $\sigma\coloneqq\sigma_{D,tH}$.
    Then, $H$ is ample, $H^2=1$, $D.H=-5/\sqrt7$, and $H.E_i=1/\sqrt7$ for $i=1,2,3$.
    Hence $\OO(-2E_1)$ and $\OO$ belong to $\TC_{D,tH}$, and
    \[
        0\to\OO(-2E_1)\to\OO\to\OO_{2E_1}\to0
    \]
    is exact in $\AC_{D,tH}$.
    Direct calculation gives the following values, where $Z=Z_{D,tH}$ and $\beta=\beta_\sigma$:
    \[
        \begin{array}{c|ccccc}
            & \OO & \OO(-2E_1) & \OO(-E_1) & \OO(-E_2) & \OO(-E_3)\\ \hline
            \Re Z & 2 & 0 & 1/2 & 5/2 & 11/2\\
            \sqrt7\,\Im Z & 5 & 3 & 4 & 4 & 4\\
            \beta & -2\sqrt7/5 & 0 & -\sqrt7/8 & -5\sqrt7/8 & -11\sqrt7/8
        \end{array}
    \]
    In particular,
    \[
        \beta(\OO(-2E_1))>\beta(\OO(-E_1))>\beta(\OO)>\beta(\OO(-E_2))>\beta(\OO(-E_3)).
    \]
    Thus $\OO(-E_1)$ detects the instability of $\OO$, as in \cref{Theorem: stability of line bundle}, but $\OO(-2E_1)$ has larger slope than each $\OO(-E_i)$.

    Destabilizing subobjects need not be line bundles.
    Let $x\in X\setminus E_1$ be a closed point and put $F\coloneqq I_x\otimes\OO(-2E_1)$.
    Since $F$ is torsion-free of rank $1$ with $\mu_H(F)=-2/\sqrt7>D.H$ and $\OO/F$ is torsion, the natural morphism $F\to\OO$ is a subobject in $\AC_{D,tH}$.
    Its central charge is $Z(F)=1+3i/\sqrt7$, so $\beta(F)=-\sqrt7/3>-2\sqrt7/5=\beta(\OO)$.
\end{eg}

\section{Stability of torsion sheaves}\label{section: stability of torsion sheaves}

In this section, we prove \cref{proposition: stability of torsion sheaf} for sheaves supported on integral curves of nonpositive self-intersection on an arbitrary smooth projective surface.
In particular, this proves \cref{lemma: stability on a negative curve}, which was used in \cref{subsection: line bundle stability higher rank}.
The argument follows the same rank reduction as in \cite[Section 5.2]{arcara_miles_2017_projectivity_of_bridgeland_moduli_spaces_on_del_pezzo_surfaces_of_picard_rank_2}.
The Hodge index theorem allows us to work in arbitrary Picard rank, and we use the full numerical wall equation at $t>0$.

Throughout this section, let $Y$ be a smooth projective surface over $\CB$, let $C\subset Y$ be an integral curve, and let $\iota:C\hookrightarrow Y$ be the inclusion.
Here, integral means reduced and irreducible; the curve need not be smooth.
Let $L$ be a torsion-free sheaf of rank $1$ on $C$, and set $S\coloneqq\iota_*L$.
Thus, $S$ is a pure one-dimensional sheaf with $\ch(S)=(0,C,q)$, where $q\coloneqq\ch_2(S)$.
All ranks of sheaves on $Y$ below are taken on the surface.
We work with real numerical divisor classes in $N^1(Y)_{\RB}$; in particular, $H^\perp$ below is taken in this numerical vector space.
Fix an ample $\RB$-divisor $H$ with $H^2=1$ and put $e\coloneqq C.H>0$.

\begin{rem}\label[rem]{remark: subobject and filtration torsion case}
    Let $F\subset S$ be a nonzero proper subobject in $\AC_{D,tH}$, and let $Q$ be its quotient in the heart.
    Taking coherent cohomology shows that $F$ is a sheaf in $\TC_{D,tH}$ and gives an exact sequence
    \[
        0\to K\to F\to S\to T_0\to0,
    \]
    where $K=H^{-1}(Q)\in\FC_{D,tH}$ and $T_0=H^0(Q)$.
    The morphism $F\to S$ is nonzero, since otherwise $F\simeq K$ would belong to both parts of the torsion pair.
    Since $S$ is pure of generic rank $1$ on the integral curve $C$, its nonzero image has zero-dimensional quotient, so $T_0$ is zero-dimensional.
    The rank $0$ and positive-rank cases are treated separately in \cref{Lemma: stability of torsion sheaf rank 0}.
\end{rem}

\begin{lem}\label[lem]{Lemma: stability of torsion sheaf rank 0}
    Let $F\subset S$ be a nonzero proper subobject in $\AC_{D,tH}$, and put $K\coloneqq\ker(F\to S)$ in $\Coh(Y)$.
    If $\rk(F)=0$, then $\beta(F)<\beta(S)$.
    If $\rk(F)>0$, then $F$ and $K$ are torsion-free of the same rank, and the subobject condition is
    \[
        \mu_{H,\max}(K)\leq D.H<\mu_{H,\min}(F).
    \]
    For a fixed morphism $F\to S$ in the positive-rank case, this condition depends only on $D.H$ and not on $t>0$ or the component of $D$ in $H^\perp$.
\end{lem}
\begin{proof}
    Consider the coherent exact sequence in \cref{remark: subobject and filtration torsion case}.
    If $\rk(F)=0$, then $K=0$ because $K$ is torsion-free.
    Thus, $F$ is a proper subsheaf of $S$, and the zero-dimensional quotient $T_0$ has positive length.
    Since $Z(F)=Z(S)+\ell(T_0)$ and $\Im Z(F)=\Im Z(S)=te>0$, we obtain
    \[
        \beta(F)=\beta(S)-\frac{\ell(T_0)}{te}<\beta(S).
    \]

    Suppose that $\rk(F)>0$.
    Since $S$ and $T_0$ have rank $0$, we have $\rk(K)=\rk(F)$.
    If the torsion subsheaf $T$ of $F$ were nonzero, then $T\to S$ would be injective because $K$ is torsion-free.
    Its cokernel is zero-dimensional, since $S$ is pure of generic rank $1$ on $C$.
    Consequently, the injection $K\to F/T$ also has zero-dimensional cokernel, so $\mu_H(K)=\mu_H(F/T)$.
    However, $K\in\FC_{D,tH}$ and $F/T\in\TC_{D,tH}$ imply $\mu_H(K)\leq D.H<\mu_H(F/T)$, a contradiction.
    Thus, $F$ is torsion-free.

    The two slope inequalities are equivalent to $K\in\FC_{D,tH}$ and $F\in\TC_{D,tH}$.
    Conversely, these conditions imply that $F\to S$ is a subobject in the heart, since its coherent cokernel is zero-dimensional.
    As $H$ and the morphism $F\to S$ are fixed, the condition depends only on $D.H$.
\end{proof}

We next compute the numerical wall equation used in the proof of the stability of $S$.

\begin{lem}\label[lem]{lemma: torsion case shape of wall}
    Let $F$ be a coherent sheaf of rank $r>0$, and write
    \[
        C=eH+W,\quad d\coloneqq c_1(F)/r=\mu H+V,\quad \gamma\coloneqq\ch_2(F)/r,
    \]
    where $H.W=H.V=0$.
    On $H^\perp$, use the norm $\|U\|^2=-U^2$ given by the Hodge index theorem, and put
    \[
        R_F(s)\coloneqq-\frac{C^2}{e^2}(s-\mu)^2
        +\frac{2(q-d.C)}{e}(s-\mu)-d^2+2\gamma.
    \]
    For $D=sH+U$ with $H.U=0$ and $t>0$, we have
    \begin{equation}\label{equation: integral curve positive volume wall}
    \frac{2}{ret}\big(\Re Z(S)\Im Z(F)-\Re Z(F)\Im Z(S)\big)
    =R_F(s)-\left\|U-V-\frac{s-\mu}{e}W\right\|^2-t^2,
    \end{equation}
    where the central charges are taken at $\sigma_{D,tH}$.
\end{lem}
\begin{proof}
    We have $\|W\|^2=e^2-C^2$ and
    \begin{align*}
        Z(S)&=D.C-q+ite,\\
        Z(F)&=r\left(-\gamma+D.d-\frac{D^2-t^2}{2}+it(\mu-s)\right).
    \end{align*}
    Substituting $D=sH+U$ and $d=\mu H+V$, we obtain
    \begin{align*}
        \frac{2}{ret}\big(\Re Z(S)\Im Z(F)-\Re Z(F)\Im Z(S)\big)
        &=-(s-\mu)^2-\|U-V\|^2-\frac{2(s-\mu)}{e}(U-V).W\\
        &\quad+\frac{2(q-d.C)}{e}(s-\mu)-d^2+2\gamma-t^2.
    \end{align*}
    Completing the square in $U-V$ and using $\|W\|^2=e^2-C^2$ gives \eqref{equation: integral curve positive volume wall}.
\end{proof}

\begin{prop}\label[prop]{proposition: stability of torsion sheaf}
    Let $Y$ be a smooth projective surface over $\CB$, and let $C\subset Y$ be an integral curve with $C^2\leq0$.
    If $L$ is a torsion-free sheaf of rank $1$ on $C$, then $\iota_*L$ is $\sigma$-stable for every divisorial stability condition $\sigma\in\divstab(Y)$.
\end{prop}
\begin{proof}
    Fix $H$ with $H^2=1$, and suppose that $S=\iota_*L$ is not stable at some $\sigma_{D,tH}$.
    Among all pairs consisting of a stability condition $\sigma_{D,tH}$ with this fixed $H$ and a nonzero proper subobject $F\subset S$ with $\beta(F)\geq\beta(S)$, choose one for which $r\coloneqq\rk(F)$ is least.
    By \cref{Lemma: stability of torsion sheaf rank 0}, we have $r>0$, and $F$ and $K\coloneqq\ker(F\to S)$ are torsion-free of rank $r$.
    Put $a\coloneqq\mu_{H,\max}(K)$, $b\coloneqq\mu_{H,\min}(F)$, $s_0\coloneqq D.H$.
    Then, $a\leq s_0<b$.

    In the notation of \cref{lemma: torsion case shape of wall}, write $D=s_0H+U_0$ with $H.U_0=0$.
    Since $\Im Z(F)>0$, $\Im Z(S)=te>0$ and $\beta(F)\geq\beta(S)$, \eqref{equation: integral curve positive volume wall} gives
    \[
        R_F(s_0)\geq\left\|U_0-V-\frac{s_0-\mu}{e}W\right\|^2+t^2>0.
    \]
    Since $C^2\leq0$, \cref{lemma: torsion case shape of wall} gives $R_F''(s)=-2C^2/e^2\geq0$.
    Thus, $R_F$ is convex, and at least one of $R_F(a)$ and $R_F(b)$ is positive.

    Choose $c\in\{a,b\}$ such that $R_F(c)>0$, and let $I\subset(a,b)$ be an open interval with an endpoint $c$ on which $R_F$ is positive.
    For $s\in I$, set
    \[
        D_s\coloneqq sH+V+\frac{s-\mu}{e}W,\quad t_s\coloneqq\sqrt{R_F(s)},
    \]
    and $\sigma_s\coloneqq\sigma_{D_s,t_sH}$.
    Then $D_s.H=s$, and $\sigma_s$ has a limit $\sigma_*=\sigma_{D_*,t_*H}$ as $s\to c$, with $t_*=\sqrt{R_F(c)}>0$.
    Since $a<s<b$, the fixed morphism $F\to S$ is a subobject in $\AC_{D_s,t_sH}$ by \cref{Lemma: stability of torsion sheaf rank 0}.
    Moreover, \eqref{equation: integral curve positive volume wall} gives $\beta_{\sigma_s}(F)=\beta_{\sigma_s}(S)$ for $s\in I$.
    The cokernel of $F\to S$ is zero-dimensional, and $S\in\TC_{D,tH}$ with $\Im Z_{D,tH}(S)=te>0$ for every $D$ and $t>0$.
    Thus, \cref{lemma: Bertram's lemma for line bundle}, applied with $B=S$, gives a nonzero proper subobject of $S$ of rank smaller than $r$ which strictly destabilizes $S$ at $\sigma_s$ for $s$ sufficiently close to $c$.
    This contradicts the minimality of $r$, since $H$ remains fixed while $D_s$ and $t_s>0$ are among the parameters over which the pair was chosen.

    Hence, $S$ is stable for every $D$ and $t>0$ with $H$ fixed.
    Since any ample $\RB$-divisor is a positive real multiple of one with square $1$, the claim follows.
\end{proof}

\begin{rem}
    Taking $Y=X$, $C=E_i$ and $L=\OO_{E_i}(k)$ in \cref{proposition: stability of torsion sheaf} proves \cref{lemma: stability on a negative curve}. The same conclusion holds for any $(-1)$-curve on any smooth projective surface.
\end{rem}

\begin{rem}
    The sign condition cannot be omitted in general.
    Let $Y=\PPT$, let $C$ be a line, and let $H$ be the hyperplane class.
    At $D=-H/2$ and $t=1/4$, the morphism $\OO_Y\to\OO_C$ is a subobject in $\AC_{D,tH}$: its coherent kernel $\OO_Y(-H)$ belongs to $\FC_{D,tH}$, whereas $\OO_Y\in\TC_{D,tH}$.
    Direct calculation gives
    \[
        Z(\OO_Y)=-\frac{3}{32}+\frac{i}{8},\quad Z(\OO_C)=\frac{i}{4},
    \]
    so $\beta(\OO_Y)=3/4>0=\beta(\OO_C)$.
    The assumption that $C$ is integral cannot be omitted in general. For two disjoint $(-1)$-curves $E_1,E_2$, the sheaf $\OO_{E_1\cup E_2}=\OO_{E_1}\oplus\OO_{E_2}$ is never stable, although $(E_1+E_2)^2=-2$.
\end{rem}

\bibliographystyle{amsalpha}
\bibliography{converted_bibtex_tyoshida}
\end{document}